\documentclass[12pt, reqno]{amsart} 
 \usepackage{amsmath,amssymb,latexsym,footnote,dsfont}
 \usepackage{graphicx,pstricks}
\usepackage{fullpage}
\usepackage{mathrsfs}
\usepackage{amsfonts}
\usepackage{amsthm}
\usepackage{amsfonts}
\usepackage[toc,page]{appendix}
\usepackage{ mathrsfs }
\usepackage{amsmath}
\usepackage[colorlinks=true, pdfstartview=FitV, linkcolor=blue, citecolor=blue, urlcolor=blue]{hyperref}
\makeatletter  
\def\@@insvline#1#2{{\setbox0\hbox{\m@th$#1\mathrm I$}  
  \rlap{\m@th$#1 \mkern 5mu  
  \vrule height.95\ht0 depth-.005\ht0 width.09\ht0 $}  
  {\mathrm #2} }}

\def\Q{\mathpalette\@@insvline{Q}}

\makeatother  
  \newtheorem{defi}{Definition}
  \newcommand{\bd}{\begin{defi}} 
  \newcommand{\ed}{\end{defi}}

  \newtheorem{lemm}[defi]{Lemma}  
  \newcommand{\bl}{\begin{lemm}}
  \newcommand{\el}{\end{lemm}} 

  \newtheorem{theo}[defi]{Theorem}
  \newcommand{\bt}{\begin{theo}}
  \newcommand{\et}{\end{theo}}

  \newtheorem{cor}[defi]{Corollary}
  \newcommand{\bc}{\begin{cor}}
  \newcommand{\ec}{\end{cor}}

  \newtheorem{pro}[defi]{Proposition}
  \newcommand{\bp}{\begin{pro}}
  \newcommand{\ep}{\end{pro}}
  
  \newcommand{\bydef}{\overset{def}{=}}

  \makeatletter
  \def\proof{\@ifnextchar[\opargproof{\opargproof[\bf Proof \hfil\\ ]}}
  \def\opargproof[#1]{\par\noindent {\bf #1 }}
  
  \makeatother
\DeclareFontFamily{OT1}{nice}{}
\DeclareFontShape{OT1}{nice}{m}{n}{<5> <6> <7> <8> <9> <10>
<12><10.95><14.4><17.28><20.74><24.88>callig15}{}
\DeclareFontFamily{U}{nice}{}
\DeclareFontShape{U}{nice}{m}{n}{<5> <6> <7> <8> <9> <10>
<12><10.95><14.4><17.28><20.74><24.88>callig15}{}
\DeclareSymbolFont{calligra}{U}{nice}{m}{n}
\DeclareSymbolFontAlphabet{\nice}{calligra}
\DeclareFontFamily{OT1}{cmdh}{}
\DeclareFontShape{OT1}{cmdh}{m}{n}{<10>cmdunh10}{}
\input umsa.fd
\input umsb.fd
\input ueuex.fd
\input ueuf.fd
\input ueur.fd
\input ueus.fd

\def\epsilon{\varepsilon}  
\def\phi{\varphi}

\newtheorem{thm}{Theorem}
\newtheorem{rmq}{Remark}
\newtheorem{lemma}{Lemma}
\newtheorem{prop}{Proposition}
\newtheorem{defin}{Definition}

\newtheorem{lemmasec}{Lemma}[section]
\newtheorem{propsec}{Proposition}[section]

\newcommand{\norm}[1]{\left\Vert #1\right\Vert}
\newcommand\blfootnote[1]{%
  \begingroup
  \renewcommand\thefootnote{}\footnote{#1}%
  \addtocounter{footnote}{-1}%
  \endgroup}
 
\thanks{This Work has been done when the author was a PhD student in the University of Nice-C\^ote d'Azur-France, under the supervision of F.Planchon and P.Dreyfuss. In particular, the author would like to thank his supervisors for the accomplished work.}
\thanks{The author would like to thank Jin Tan for the fruitful discussion and the remarks on the subject.}
\date{\today}

\author[H. Houamed]{Haroune Houamed}
\address{CNRS, LJAD, Université Co\^te d'Azur\\ Département de Mathématiques\\ Nice,  France}
\email{haroune.houamed@univ-cotedazur.fr}

\title[Wellposedness for the Electron Inertial Hall-MHD]{Well-posedness and long time behavior for the electron inertial Hall-MHD system in Besov and Kato-Herz spaces}
\begin{document}

\begin{abstract}  In this paper, we study the wellposedeness of the Hall-magnetohydrodynamic system augmented by the effect of electron inertia. Our main result consists of generalizing the wellposedness one in \cite{Zhao} from the Sobolev context to the general Besov spaces and Kato-Herz space. Then we show that we can reduce the required regularity of the magnetic field in the first result modulo an additional condition on the maximal time of existence. Finally, we show that the $\widehat{L}^p$ (and eventually the $L^p$) norm of the solution $(u,B,\nabla \times B)$ associated to an initial data in $ \widehat{B}^{\frac{3}{p}-1}_{p,\infty}(\mathbb{R}^3)$, is controlled by $t^{-\frac{1}{2}(1-\frac{3}{p})}$, for all $p\in (3,\infty)$, which provides a polynomial decay to zero of the $\widehat{L}^p$  norm of the solution.
\end{abstract}
\maketitle
\blfootnote{{\it keywords:}
Hall-MHD system, MHD equations, Littlewood-Paley, critical spaces, Long time behavior.}  
\tableofcontents

\section{Introduction}
\subsection{Overview and connection with other models}
In this paper we consider the incompressible 3D electron inertia Hall-MHD equations derived from the two fluid model of ion and electron
\begin{equation}\label{1.1}
\tag{1} \left\{\begin{array}{l}
 \partial_t u + u \cdot \nabla u - \mu \Delta u + \nabla P=j\times  \big[ B- \delta (1-\delta)\lambda^2 \Delta B \big],\\
 \partial_t B' - \eta\Delta B' +  (1-\delta)\lambda^2 \mu _e \Delta^2 B =  \nabla \times \big\{ \big[ u-(1-\delta)\lambda j \big] \times B' \big\}- \lambda \mu _e \Delta( \nabla \times u), \\
    j= \frac{c}{4 \pi} \nabla \times B,\\
   B'= B- \delta(1-\delta )\lambda^2 \Delta^2 B- \delta \lambda (\nabla \times u),\\
  div\, u = div\, B = 0 \\
  (u_{|t=0},B_{|t=0})= (u_0,B_0).
\end{array}\right.
\end{equation}
Here, $u$ is the hydrodynamic velocity, $B$ the magnetic field, the scalar function $P$ denotes the pressure which can be recovered, due the incompressiblity condition, from the relation $ -\Delta P = \nabla \cdot \big( u\cdot\nabla u - j \times [ B- \delta(1-\delta)\lambda^2 \Delta B] \big)$. $c$ is the speed of light and $j$ denotes the electric current density. $\mu_e$ is the kinematic viscosity of the electron fluid, and if we denote by $\mu_i$ that of the ion fluid, then the kinematic viscosity $\mu$ will be $\mu_i + \mu_e$. The very small parameter $\delta$ is given by $\delta= m_e/M$, where $M=m_e+m_i$, which is the sum of the mass of an ion and of an electron.

If we denote by $e,\, n, \, L_0$ the charge, the number density and the length scale respectively, and if we define $w_M \bydef (4\pi e^2 n / M)^{\frac{1}{2}}$, then $\lambda$ will be $\lambda= c/ w_M L_0$.

As pointed out in the introduction of \cite{Zhao}, the system above can be seen as a full two-fluids MHD description of a completely ionized hydrogen plasma, retaining the effects of the Hall current, electron pressure and electron inertia.
For more details about the derivation of the system, we refer the reader to  \cite{abdelhamid,Zhao}.

One may notice that if we neglect the presence of the electrons in the equations, that is if we set $(\delta,\mu_e)=(0,0)$ in \eqref{1.1}, then it becomes the so-called Hall-MHD system as follows
\begin{equation}\label{hall-mhd}
\tag{$Hall.MHD$} \left\{\begin{array}{l}
 \partial_t u + u \cdot \nabla u - \mu \Delta u + \nabla P=j\times   B ,  \\
 \partial_t B  - \eta\Delta B  +    =  \nabla \times \big\{ ( u- \lambda j ) \times B  \big\},  \\
    j= \frac{c}{4 \pi} \nabla \times B,\\
  div\, u = div\, B = 0, \\
  (u_{|t=0},B_{|t=0})= (u_0,B_0).
\end{array}\right.
\end{equation}
Hence, at least formally, the solution $(u,B)$ to \eqref{hall-mhd} can be seen as the limit of the solution of \eqref{1.1} when $(\delta,\mu_e)$ tends to $(0,0)$. 

Several works have been devoted to the study of the systems \eqref{1.1}, \eqref{hall-mhd} and to the classical MHD as well \footnote{The term $\lambda \nabla \times( j\times B)$ is known in the literature as the hall-term. The classical MHD system is then the case where $\lambda=0$.}. One may see \cite{Zhao,Danchin-Jin,Liu-Jin,Giga-Yoshida,Giga-Ibarhim-AL,Agapito-Schonbek,Chae-Wan-Wu,Wan-Zhou} for more details.   

The main purpose of this paper is to deal with the system \eqref{1.1}. It is easy to check, by making use of the equation satisfied by $(\nabla \times u)$, that \eqref{1.1} can be simplified into the following system 
\begin{equation}\label{H.MHD}
\tag{$\widetilde{MHD}$} \left\{\begin{array}{l}
 \partial_t u + u \cdot \nabla u - \Delta u + \nabla P=(\nabla \times B)\times H\\
 \partial_t H - \Delta H +2 \nabla \times \big( (\nabla \times B)\times H \big)= \nabla \times \big( u \times H\big) + \nabla \times \big( (\nabla \times B)\times (\nabla \times u ) \big)\\
   H= (Id-\Delta) B\\
  div\, u = div\, B = 0 \\
  (u_{|t=0},B_{|t=0})= (u_0,B_0),
\end{array}\right. 
\end{equation}
where we omit the constants which will play no significant role in the proofs of our results. Hence, all the constants in our results below will be modulo $ \mu,\eta, \delta$ and $\lambda$. 

 In a first stage, we aim at generalizing some of the results in \cite{Zhao} from the Hilbert spaces $\dot{H}^s(\mathbb{R}^3)$ to the general Besov spaces of the form $\dot{B}^{s}_{p,r}(\mathbb{R}^3)$. To the best of our knowledge, the wellposedness of \eqref{1.1} in general  \textit{critical} Besov spaces has not been proved before this work\footnote{We will explain in the sequel what do we mean by critical spaces to \eqref{1.1}.}. In the recent work \cite{Zhao}, the authors consider initial data in spaces of the type $H^{s}(\mathbb{R}^3)$. It is then interesting to check if similar results can be proved as well in the context of general Besov spaces having the same scaling, which is the issue treated by our Theorem \ref{thm1} below.

If we try to deal with \eqref{H.MHD} as it is, the structure of some of the nonlinear terms will prevent us from establishing the wellposedness for all \footnote{$p$ refers here to the integrability with respect to space variables in the Besov spaces $\dot{B}^s_{p,r}$} $p \in [1,\infty[$. The issue will be in fact at the level of estimating the remainder terms in Bony's decomposition. Also, it is worth noting that \eqref{H.MHD} does not have a scaling invariance structure as the classical Navier-Stokes equations or the classical MHD equations.  

We recall that, the following spaces 
$$ \dot{B}^{\frac{3}{m}-1}_{ m,n}(\mathbb{R}^3) \hookrightarrow  \dot{H}^\frac{1}{2}(\mathbb{R}^3)\hookrightarrow \dot{B}^{\frac{3}{p}-1}_{ p,r}(\mathbb{R}^3) \hookrightarrow  \dot{B}^{\frac{3}{q}-1}_{ q,r'}(\mathbb{R}^3)\hookrightarrow  \dot{B}^{ -1}_{ \infty,\infty}(\mathbb{R}^3) ,$$
for all $m,n,p,q,r,r'$ satisfying
$$   m\leq  2 \leq p\leq q\leq\infty , \;n\leq  2 \leq r \leq r'\leq\infty ,$$

are critical for the $3$D Navier-Stokes equations. On the other hand, the spaces
$$ \dot{B}^{\frac{3}{m} }_{ m,n}(\mathbb{R}^3) \hookrightarrow  \dot{H}^\frac{3}{2}(\mathbb{R}^3)\hookrightarrow \dot{B}^{\frac{3}{p} }_{ p,r}(\mathbb{R}^3) \hookrightarrow  \dot{B}^{\frac{3}{q} }_{ q,r'}(\mathbb{R}^3)\hookrightarrow  \dot{B}^{  0}_{ \infty,\infty}(\mathbb{R}^3) $$
are critical for the following $3$D system
$$ \partial_t B- \Delta B = - \nabla \times\big((\nabla \times B) \times B \big).$$
Based on this observation, one may ask the following question:  

\textbf{Question:} Can we solve \eqref{1.1} with initial data $(u_0,B_0)$ small enough in the above spaces?

 As mentioned in \cite{Zhao} for \eqref{hall-mhd}, the difficulty in establishing the wellposedness result under the condition that 
 $\norm{u_0}_{\dot{H}^\frac{1}{2}}+\norm{B}_{\dot{H}^\frac{3}{2}}$ being small enough comes from the nonlinear term $\nabla \times (u\times B)$. Up to our knowledge, all the existing results require additional smallness assumptions on $\norm{u_0}_{\dot{H} ^{\frac{1}{2}+\varepsilon}}+\norm{B_0}_{\dot{H} ^\frac{3}{2}}$ as in \cite{Wan-Zhou}, or $\norm{u_0}_{\dot{B} ^{\frac{3}{p}-1}_{p,1}}+\norm{B_0}_{\dot{B} ^{\frac{3}{p}-1}_{p,1}}+\norm{\nabla \times B_0}_{\dot{B} ^{\frac{3}{p}-1}_{p,1}}$ as in \cite{Danchin-Jin,Liu-Jin}. 
 
 In our analysis, we find that \eqref{H.MHD} is easier to deal with due to the flexibility of the operator $(Id-\Delta)^{-1}$ that appears in the Duhamel formula\footnote{See next subsection for the Duhamel formula associated to \eqref{H.MHD}, and Proposition \ref{prop.laplacian} to understand what we mean by the flexibility of this operator $(Id-\Delta)^{-1}$.} of \eqref{H.MHD}. This remarkable feature is a key ingredient in our proofs, and it is in fact one of the reasons why, in our Theorem \ref{thm2} below, we can omit the smallness condition on $\norm{B_0}_{\dot{B} ^{\frac{3}{p}-1}_{p,1}}$ modulo an additional condition on the maximal time of existence. This provides a partial answer to the question above in the case of the Electron Inertial Hall-MHD. 
 
On the other hand, to give a sense to the critical spaces in the case of \eqref{hall-mhd}, the authors in \cite{Danchin-Jin} considered the following ``equivalent" system 
\begin{equation}\label{hall-mhd2}
\tag{$\widetilde{Hall.MHD}$} \left\{\begin{array}{l}
 \partial_t u + u \cdot \nabla u - \mu \Delta u + \nabla P=j\times   B ,  \\
 \partial_t B  - \eta\Delta B       =  \nabla \times \big\{ ( u- \lambda j ) \times B  \big\},  \\
 \partial_t j  - \eta\Delta j       =  \nabla \times\big(\nabla \times \big\{ ( u- \lambda j ) \times curl^{-1}j  \big\}\big),  \\
 j= \nabla \times B,\\
  div\, u = div\, B = 0 ,\\
  (u_{|t=0},B_{|t=0})= (u_0,B_0),
\end{array}\right.
\end{equation}
 where the third equation is obtained from the second one by applying the curl operator. Above, the $curl^{-1}$ is defined as 
 $$ curl^{-1}j\bydef \mathcal{F}^{-1}\bigg(\frac{i \xi \times \widehat{j}}{|\xi|^2}\bigg).$$
 The system \eqref{hall-mhd2} has then a \textit{scaling invariance}, which is exactly the same as that of the $3D$ Navier-Stokes equations and the classical MHD system.

Hence, taking into account the fact that \eqref{hall-mhd} can be seen as the limit system of \eqref{1.1} when $(\delta,\mu_e)$ vanishes, a space will be called critical to \eqref{1.1} if it is critical to \eqref{hall-mhd} in the sense explained above and in \cite{Danchin-Jin}.  

We conclude this subsection by pointing out an interesting issue: One may notice that in our first result below for \eqref{1.1} (see Theorem \ref{thm1}) we will be considering data $u_0,B_0,\nabla \times  B_0$ in $\dot{B}^{\frac{3}{p}-1}_{p,r}$, where $r$ can take any value in $[1,\infty]$. On the other hand, for \eqref{hall-mhd} and as mentioned in \cite{Danchin-Jin}, proving the wellposedness in such spaces remains an open question for the time being, except for the case $p=2$ or $r=1$. As explained in \cite{Danchin-Jin}, one may benefit from some cancellations in the case $p=2$ to get ride of the most nonlinear term in the equation of $J$. Whereas, in the case $p\neq2$, the index $r$ is needed to be equal to $1$ because, at some point, they need to estimate a product of two functions in the critical Besov space $\dot{B}^{\frac{3}{p}}_{p,r}$.\footnote{ Fore more details on that we suggest to see in particular the estimates $(3.5)$ and $(3.6)$ in \cite{Danchin-Jin}.} This can be completely  avoided in our case, thanks again to the flexibility of the operator $(Id-\Delta)^{-1}$. 

 We recall that, in the present work, we are not interested in studying the asymptotic limit of the solutions to \eqref{1.1} when $(\delta,\mu_e)$ vanish. But  to be more precise and rigorous, we should point out that, if we keep all the constants in our analysis, the operator $(Id-\Delta)^{-1}$ will be in fact $(Id-\delta(1-\delta)\lambda\Delta)^{-1}$, which means that, when $\delta$ vanishes, we do not believe that we can give an answer to the open question in \cite{Danchin-Jin} by studying the limit (as $(\delta,\mu_e)$ vanish) of our solution given by Theorem \ref{thm1}.
\subsection{Reformulation of the equations}
We recall again that, for simplification, we will prove our theorems for \eqref{H.MHD} instead of \eqref{1.1}, whereas we should point out that all our results below hold as well for \eqref{1.1}. 

Let us then first rewrite the system \eqref{H.MHD} in an appropriate form. To do so, we recall some vectorial concepts.\\
For $U,V $ two divergence-free vector fields on $\mathbb{R}^3$, we have
\begin{equation}\label{vect1}
\nabla \times ( U \times V) = V \cdot \nabla U - U \cdot \nabla V,
\end{equation}
\begin{equation}\label{vect2}
(\nabla \times U) \times U = U \cdot \nabla U - \frac{1}{2} \nabla |U|^2,
\end{equation}
\begin{equation}\label{vect3}
 V\times (\nabla \times U)   +U\times(\nabla \times V)   =- \nabla \times ( U \times V)   - 2 U \cdot \nabla V + \nabla (U\cdot V),
\end{equation} 
\begin{equation}\label{vect4}
\nabla \times (\nabla  \times U) = -\Delta U .
\end{equation}
If we denote $J \bydef \nabla \times B$, then according to \ref{vect2} and \ref{vect4}, we obtain $$ ( \nabla \times B) \times (B - \Delta B) =B \cdot \nabla B - J\cdot \nabla J + \nabla \bigg( \frac{|B|^2 - |J|^2}{2} \bigg). $$
On the other hand, we have 
$$ u \times H  +  (\nabla \times B)\times (\nabla \times u ) =u \times B  +  u\times (\nabla \times J)+ J\times (\nabla \times u ). $$
Thus according to \eqref{vect3}, we infer that
\begin{equation}
 u \times H  +  (\nabla \times B)\times (\nabla \times u ) = u \times B  - \nabla \times (J \times u) - 2 J \cdot \nabla u + \nabla ( J \cdot u).
\end{equation}
Therefore, \eqref{H.MHD} can be written as follows
\begin{equation}\label{SS}
\left\{\begin{array}{l}
 \partial_t u - \Delta u =B \cdot \nabla B - J\cdot \nabla J -  u \cdot \nabla u  + \nabla \tilde{P} \\
 \partial_t B - \Delta B  =   (Id-\Delta)^{-1}\nabla \times \Theta \\
  \partial_t J - \Delta J  =  -\Delta (Id-\Delta)^{-1}  \Theta + \nabla (Id-\Delta)^{-1}(\nabla \cdot \Theta) \\
  div\, u = div\, B =div J= 0 \\
  (u ,B,J)_{|t=0}= (u_0,B_0,\nabla \times B_0),
\end{array}\right.  \tag{$S_1$}
\end{equation}
where, $$ \tilde{P}\bydef -p + \frac{|B|^2 - |J|^2}{2},$$
and
$$\Theta \bydef  u \times B -2 B \cdot \nabla B - J\cdot \nabla J  - \nabla \times (J \times u) - 2 J \cdot \nabla u + \nabla (J\cdot u).$$
One may notice that, in the expression of $\Theta $ above, the term $u\times B$ is not at the same level of scaling compared to the rest of quantities, besides that the non-homogeneous behavior of the operator $(Id-\Delta)^{-1}$ destroys any chances for \eqref{SS} to have a scaling invariant structure as the classical Navier-Stokes has for example, or as in the case of the system \eqref{hall-mhd2} studied in \cite{Danchin-Jin}.\\
The justification of the choice of the functional framework in which we do our analysis is then made by working in critical spaces as explained in the previous subsection.\\

Let us now  write down Duhamel's formula corresponding to \eqref{SS}. Let $\mathbb{P}$ be the Leray projector, we write
 $$ \mathcal{Q} (u,v ) \bydef  \nabla \cdot\big( u \otimes v  \big) , $$
 $$\mathcal{P}(u,v ) \bydef  u\times v ,$$
 $$\mathcal{R}(u,v) \bydef \nabla \times ( u \times v), $$
 for a three dimensional vectors $K= (K_1,K_2,K_3)$, $L=(L_1,L_2,L_3)$, we define
$$ \Gamma (K,L)\bydef \begin{pmatrix}
 \mathcal{Q} (K_2, L_2)- \mathcal{Q} (K_1, L_1) - \mathcal{Q} (K_3, L_3)  \\
 \mathcal{P}(K_1 ,L_2  ) - \mathcal{R}( k_3,L_1) -    \mathcal{Q} (K_3, L_3) -2 \mathcal{Q} (K_2, L_2)-2  \mathcal{Q} (K_3, L_1)\\
  \mathcal{P}(K_1 ,L_2  ) - \mathcal{R}( k_3,L_1) -    \mathcal{Q} (K_3, L_3) -2 \mathcal{Q} (K_2, L_2)-2  \mathcal{Q} (K_3, L_1)
\end{pmatrix},$$
\begin{equation}\label{omega expression}
 \Omega (K,L) \bydef \mathbb{P}\begin{pmatrix}
\Gamma_1 (K,L)\\
(Id-\Delta)^{-1} \nabla \times \Gamma_2 (K,L)\\
-\Delta( Id-\Delta)^{-1} \Gamma_3 (K,L)
\end{pmatrix} ,
\end{equation}
and
\begin{equation}\label{zeta definition}
\zeta (K,L)\bydef \mathcal{K} \Omega (K,L),
\end{equation}
where $$\mathcal{K} \varphi(t,\cdot)= \int_0^t e^{(t-s)\Delta} \varphi (s,\cdot)ds.$$
Finally, if we write $ \mathcal{U} \bydef (u,B,J)$, then \eqref{SS} is equivalent to
\begin{equation}\label{S zeta}
\left\{\begin{array}{l}
 \partial_t \mathcal{U} - \Delta \mathcal{U} = \zeta(\mathcal{U}, \mathcal{U}) \\
  div\, u = div\, B =div J= 0 \\
   \mathcal{U} _{|t=0}= \mathcal{U}_0
\end{array}\right. \tag{$S_{\zeta}$}
\end{equation}
\begin{rmq}
As mentioned in a paper of I.Gallagher, D.Iftime and F.Planchon \cite{Planchon-assympto}  for the classical Navier Stokes equations, the theory of weak solutions to the Navier Stokes equations is related to the special structure of the equation, namely to the energy inequality, while the Kato's approach is more general and can be applied to more general parabolic or dispersive equations, this work is an example among many. The main issue here consists in writing the equations in an appropriate form in order to be able to adapt the techniques used for the classical Navier-Stokes \cite{Chemin1,Planchon-assympto,Leray1,Danchin2}.
\end{rmq}
\section{Preliminaries and statement of the main results}
\subsection{Functional framework and notations}
In this subsection, we set up some notations, definitions and the functional spaces in which we do our analysis. We shall provide some properties of these spaces in Appendix part of the paper.\\

Let us begin by recalling the Littlewood-Paley decomposition and the associated Besov spaces
\noindent Let $(\psi,\varphi)$ be a couple of smooth functions with value in $[0,1]$ satisfying: 
\begin{alignat*}{2}
&\text{Supp } \psi \subset \{ \xi \in \mathbb{R} : |\xi| \leq \frac{4}{3}\}, 
\quad &&\text{Supp } \varphi \subset \{\xi \in \mathbb{R} :\frac{3}{4} \leq |\xi| \leq \frac{8}{3} \} \\
&\psi(\xi) + \sum_{q\in \mathbb{N}} \varphi(2^{-q}\xi) = 1 \;\; \forall \xi \in \mathbb{R}, 
\quad 
&&\sum_{q\in \mathbb{Z}} \varphi(2^{-q}\xi) = 1 \;\; \forall \xi \in \mathbb{R}\backslash \{0\}.
\end{alignat*}
Let $a$ be a tempered distribution, $\widehat{a}=\mathcal{F}(a)$ its Fourier transform and $\mathcal{F}^{-1}$ 
denotes the inverse of $\mathcal{F}$. We define the homogeneous dyadic blocks $\Delta_q$ by setting:
\begin{alignat*}{2}
&\Delta _q a \bydef \mathcal{F}^{-1}\big(\varphi (2^{-q}|\xi | \widehat{a} ) \big)\;  , 
\quad && S_q  \bydef  \sum_{j<q} \Delta_j  \ \forall q \in \mathbb Z.
\end{alignat*}
Although the previous sections, we used the Bony's decomposition which reads as follows, for tempered distributions $u$ and $v$, we have
$$uv = T_uv + T_v u + R(u,v),$$
with
$$T_uv \bydef \sum_{j\in \mathbb{Z}} S_{j-1} u \Delta_j u, \;\;\;R(u,v) \bydef \sum_{j\in \mathbb{Z}}  \tilde{\Delta}_j u \Delta _j v,$$
where $\tilde{\Delta}_j \bydef \sum_{i\in \{-1,0,1\}} \Delta_{j+i}.$\\
According to the support properties above, we have
$$\Delta _q T_uv = \Delta_q \sum_{ j\sim q} S_{j-1} u \Delta_j u $$
$$\Delta _q R(u,v) = \Delta_q \sum_{j\geq q+N_0}  \tilde{\Delta}_j u \Delta _j v, $$
for some $N_0 \in \mathbb{Z}.$\\
Based on the dyadic decomposition presented above, we recall the definition of the usual Besov spaces on $\mathbb{R}^d$ and the Chemin-Lerner spaces defined on $\mathbb{R}^+\times \mathbb{R}^d$. 
\begin{defin}
Let $s $ be a real number and $p,r$ be in $[1,+\infty]$, we define the space $ B^s_{p,r}(\mathbb{R}^d) $ as the space of tempered distributions $u $ in $\mathcal{S}(\mathbb{R}^d)$ such that 
$$
\norm u _{\dot{B}^s_{p,r}}\bydef \norm {  2^{js} \norm { \Delta_j  u}_{L^p}}_{  \ell_j^{r}(\mathbb{Z})) } < \infty 
$$
And for $\rho\in [1,\infty]$, the space $\tilde{L}^\rho(\dot{B}^s_{p,r}) $ is the space of the tempered distributions $f$ in $\mathcal{S}(\mathbb{R}^+\times \mathbb{R}^d)$ such that
$$\norm f _{\tilde{L}_T^\rho(\dot{B}^s_{p,r})}\bydef   \norm{2^{js} \norm { \Delta_j  u}_{L^\rho_T L^p}}_{  \ell_j^{r}(\mathbb{Z}) }   < \infty  $$
\end{defin}
Next, we recall the definition of Kato spaces, then we introduce the Kato-Herz and the Fourier-Herz spaces used in Theorem \ref{thm3}, for more details about the Fourier-Herz spaces the reader can see for instance \cite{Chikami}.
\begin{defin}[Kato spaces]
 Let $ p$ be in $ [1,\infty]$, $\sigma \in \mathbb{R}_+^*$, we define the space $K^\sigma_{p,r}(T)$ (or simply $K^\sigma_{p,r} $ when $T=\infty$), as the space of functions $u$ on $\mathbb{R}^+ \times \mathbb{R}^d$, such that 
 $$ \norm u _{K^\sigma_{p,r}(T)} \bydef \norm { t^{\frac{\sigma}{2}} \norm {u(t,\cdot)}_{L^p(\mathbb{R}^d)}}_{L^r\big((0,T); \frac{dt}{t}\big)}$$
 the case $r=\infty$, we simply denote $K^\sigma_{p,\infty} = K^\sigma_{p }$, such that
 $$ \norm u_{K^{\sigma}_{p }(T)} \bydef \underset { t\in (0,T)}{\sup} \big  \{  t^{\frac{\sigma}{2}} \norm {u(t,\cdot)}_{L^p(\mathbb{R}^d})\big   \} < \infty$$
\end{defin}
\begin{defin}[\textit{Fourier-Herz and Kato-Herz spaces}]
Let $ p,r$ be in $ [1,\infty]$, $(s,\sigma) \in \mathbb{R}_+^* \times\mathbb{R}$,
we define the Fourier-Herz space $\widehat{B}^s_{p,r}(\mathbb{R}^d)$ as the space of tempered distribution $w$ on $\mathbb{R}^d$ such that
$$ \norm u _{\widehat{B}^s_{p,r}}\bydef \norm {  2^{js} \norm { \widehat{\Delta_j  u}}_{L^{p'}}}_{  \ell_j^{r}(\mathbb{Z})) } < \infty, $$
 and we define $\widehat{K}^\sigma_{p,r}(T)$ (or simply $\widehat{K}^\sigma_{p,r} $ when $T=\infty$), as the space of functions $u$ on $\mathbb{R}^+ \times \mathbb{R}^d$, such that 
 $$ \norm u _{\widehat{K}^\sigma_{p,r}(T)} \bydef \norm { t^{\frac{\sigma}{2}} \norm {\widehat{u}(t,\cdot)}_{L^{p'}(\mathbb{R}^d)}}_{L^r\big((0,T); \frac{dt}{t}\big)},$$
 the case $r=\infty$, we simply denote $\widehat{K}^\sigma_{p,\infty} = \widehat{K}^\sigma_{p }$, such that
 $$ \norm {u}_{\widehat{K}^{\sigma}_{p }(T)} \bydef \underset { t\in (0,T)}{\sup} \big  \{  t^{\frac{\sigma}{2}} \norm {\widehat{u}(t,\cdot)}_{L^{p'}(\mathbb{R}^d})\big   \} < \infty$$
\end{defin}
\begin{rmq}
In terms of the scaling, the  Fourier-Herz space $\widehat{B}^s_{p,r}$ (resp. the  Kato-Herz space $\widehat{K}^s_{p,r}$) has the same scale as the usual Besov $B^s_{p,r}$ (resp. the usual Kato $K^s_{p,r}$).
\end{rmq}

Along this work, we will be using the following notations
\begin{itemize}
\item  For $A,B$ two real quantities, $A\lesssim B$ means $A\leq c B$, for some $c>0$ independent of $A$ and $B$.
\item $(c_{j,r})_{j\in\mathbb{Z}} $ will be a sequence satisfying $ \sum_{j\in \mathbb{Z}} c_{j,r}^r  \leq 1$. This sequence is allowed to differ from a line to line, also let us point out that, due to the embedding $\ell^r(\mathbb{Z}) \in \ell ^\infty (\mathbb{Z})$, we will often use the inequality $c_{j,r}^2 \subset c_{j,r}$.
\item The $\widehat{L}^p$ norm of $u$ is given by
$$\norm {u}_{\widehat{L}^p} \bydef \norm {\widehat{u}}_{L^{p'}}, $$    where  $p' $ is the usual conjugate of  $ p,$  and$ \widehat{u}$  denotes the Fourier transform of $u.$
\item We use the notation
$$\mathscr{L} (\dot{B}_{p,r}^{s }) \bydef \displaystyle\underset{{\rho \in[1,\infty]}}{\bigcap}\tilde{L}^\rho(\mathbb{R}^+;\dot{B}^{s +\frac{2}{\rho}}_{p,r}),$$
and  for $T>0$, $$\mathscr{L}_{T} (\dot{B}_{p,r}^{s }) \bydef \displaystyle\underset{{\rho \in[1,\infty]}}{\bigcap}\tilde{L}^\rho([0,T];\dot{B}^{s +\frac{2}{\rho}}_{p,r}).$$
\end{itemize}
Let us now define what we mean by a solution to \eqref{S zeta} in this paper
\begin{defin}\label{definition solution}
Let $T>0$, and $\mathcal{U}_0 $  be given in some Banach space $\mathcal{X} $, we say that $\mathcal{U} $ is a solution to \eqref{S zeta} on $(0,T)$ if $\mathcal{U}  \in L^\infty_{loc}([0,T]; \mathcal{X}  )$ and satisfies, for a.e $t\in [0,T]$
 $$
 \displaystyle \mathcal{U}(t,\cdot) =  e^{t\Delta}\mathcal{U}_0(\cdot) + \zeta(\mathcal{U},\mathcal{U}), \;\; \text{ in }  \mathcal{X}  .$$
\end{defin}
\subsection{Statement of the main results}
\noindent The authors in \cite{Zhao} proved the wellposedness of \eqref{H.MHD} under the condition $\norm {u_0}_{\dot{H}^\frac{1}{2}} + \norm {B_0}_{\dot{H}^\frac{1}{2}} + \norm {B_0}_{\dot{H}^\frac{3}{2}} $ being small enough. Our first result consists of generalizing this last one to the Besov context, it reads as follows
\begin{thm}\label{thm1}
Let $p\in [1, \infty), \; r\in [1,\infty]$ and $\mathcal{U}_0=(u_0,B_0, \nabla \times B_0)$ be in  $\dot{B}^{\frac{3}{p}-1}_{p,r}(\mathbb{R}^3) $.  \\
There exists $c_0>0$ such that, if $$\norm {\mathcal{U}_0}_{\dot{B}^{\frac{3}{p}-1}_{p,r}}< c_0,$$
then \eqref{S zeta} has a unique global solution $\mathcal{U}$ in $\mathscr{L} (\dot{B}^{\frac{3}{p}-1}_{p,r} )$, with
$$\norm {\mathcal{U}} _{\mathscr{L} (\dot{B}^{\frac{3}{p}-1}_{p,r} )} < 2c_0.$$
\end{thm}
\begin{rmq}
One may show that the solution, in the case $r<\infty$, is continuous in time with value in $\dot{B}^{\frac{3}{p}-1}_{p,r}$, while in the case $r=\infty$ it is just weakly-continuous in time.
\end{rmq}
\begin{rmq}
 One may prove a local in time wellposedeness for large initial data, by slightly modifying the proof of Theorem \ref{thm1}, we will give some details about that in Corollary \ref{corollary local existence}.
\end{rmq}
\noindent In terms of the required regularity, in Theorem \ref{thm1} we ask for the initial data of $B$ to be in $\dot{B}^{\frac{3}{p}-1}_{p,r}(\mathbb{R}^3) \cap \dot{B}^{\frac{3}{p} }_{p,r}(\mathbb{R}^3) $. 

It is worth noting that, it is because of the two non linear terms $u\times B$ in the equation of $B$, and $ B \cdot \nabla B$ in the equation of $u$, that we do not know how to prove an analogous result to Theorem \ref{thm1}, starting from initial data $B_0$ only in $\dot{B}^{\frac{3}{p}}_{p,r}$. 

However, in the case $r=1$, we will prove that a small enough ``compared to the maximal time of existence $T^*$" initial data $B_0$ in $\dot{B}^{\frac{3}{p}}_{p,1}$ should generate a unique solution, at least up to time $T^*$. More precisely, we will prove
\begin{thm}\label{thm2}
Let $T>0$, $p\in [1,\infty)$ and $(u_0,B_0)$ be two divergence-free vector fields in $\dot{B}^{\frac{3}{p}-1}_{p,1}(\mathbb{R}^3) \times \dot{B}^{\frac{3}{p} }_{p,1}(\mathbb{R}^3)$, there exists $ c_0>0$ such that if
$$ \norm {u_0}_{\dot{B}^{\frac{3}{p}-1}_{p,1}} + (2+T) \norm {B_0}_{\dot{B}^{\frac{3}{p}}_{p,1}} < c_0,$$
then \eqref{S zeta} has a unique solution $ (u,B)$ on $(0,T)$ with
$$ \norm {u}_{\mathscr{L}_{T} (\dot{B}^{\frac{3}{p}-1}_{p,1}  )} + (2+T)\norm {B}_{\mathscr{L}_{T} ( \dot{B}^{\frac{3}{p}}_{p,1} )} < 2c_0.$$
\end{thm}
The question of the behavior, for large time, of the solution obtained in Theorem \ref{thm1} can be established along ``approximately" the same lines as those used, for instance, for the 3D Navier-Stokes equations in \cite{Planchon-assympto}. That is to say, it should be possible to prove that $ \norm {\mathcal{U}(t)}_{\dot{B}^{\frac{3}{p}-1}_{p,r}}$ tends to zero as $t$ tends to infinity.

It is also well known for the Navier Stokes equations (see the appendix of \cite{Planchon-assympto}), that for an initial data $u_0 \in \dot{B}^{\frac{3}{p}-1}_{p,r}(\mathbb{R}^3)$, $p\in (3,\infty)$, the $L^\infty $ norm of the velocity also decays to zero at infinity. More precisely it is controlled by $Ct^{-\frac{1}{2}}$.

 The proof of this result relies on the fact that the bi-linear operator in Duhamel's formula acts well on the Kato's space, together with the fact that we can iterate, as much as we want, the solution $u$ in Duhamel's formula in order to obtain a solution of the form of a sum of some $N$ multi-linear terms of $e^{t\Delta}u_0$, and a more regular remainder term $r_{N+1}$ which is unique in $L^\infty_t(L^3)$. 
 
 A priori, this approach should work as well in our case, but we will not enter into these details in this thesis.
 
In contrast to that, we will treat the case of initial data in the Herz-space $\widehat{B}^{\frac{3}{p}-1}_{p,r}(\mathbb{R}^3)$, and we will give some details in the case $r=\infty$ as an example. More precisely, we will prove
\begin{thm}\label{thm3}
Let $p\in (3,\infty)$ and $u_0,B_0 $ be two divergence free vector fields. There exists $c_0>0$ such that if
$$ \norm {e^{t\Delta}u_0}_{\widehat{K}_p^{1-\frac{3}{p}}}+\norm {e^{t\Delta}B_0}_{\widehat{K}_p^{1-\frac{3}{p}}}+\norm {e^{t\Delta}(\nabla \times B_0)}_{\widehat{K}^{1-\frac{3}{p} }_p}<c_0,$$
then \eqref{S zeta} has a unique global solution $\mathcal{U}=(u,B,\nabla\times B)$ in $\widehat{K}_p^{1-\frac{3}{p}} $ satisfying
$$\norm {\mathcal{U}(t,\cdot)}_{L^p}\lesssim t^{-\frac{1}{2}(1-\frac{3}{p})}.$$
\end{thm}
\begin{rmq}
By virtue of the characterization \eqref{caracterisation kato-besov}, we can show that the  solution obtained in Theorem \ref{thm3} is also in $\mathscr{L}(\widehat{B}_{p,\infty}^{\frac{3}{p}-1 })$, by following approximately the same steps in the proof of Theorems \ref{thm1} and \ref{thm3}. In fact for $\mathcal{U}_0 \in\widehat{B}^{\frac{3}{p}-1}_{p,r}(\mathbb{R}^3) $, small enough, we may construct a unique global solution in $\mathscr{L}(\widehat{B}^{\frac{3}{p}-1}_{p,r})$ by proceeding as in the proofs we will show in the next section. The details are left to the reader.
\end{rmq}
\noindent The rest of the paper is organized as follows: In section three we will prove, respectively, the wellposedeness in Theorem \ref{thm1} and Theorem \ref{thm2}. Then we provide some details about the proof of the wellposedeness in Kato-Herz space and the decay property described in Theorem \ref{thm3}. Finally, the Appendix is devoted to the some useful technical results.
\section{Proof of the main Theorems}
\subsection{Wellposedness in general ''critical" Besov spaces: Proof of Theorem \ref{thm1}}
The main key to prove Theorem \ref{thm1} is the following proposition
\begin{prop}\label{prop.continuity}
Let $(p,r)$ be in $[1, \infty)^2$, $u,v $ in $\mathscr{L}(\dot{B}^{\frac{3}{p}-1}_{p,r})$, $ \mathcal{Q}$, $ \mathcal{R}$, $ \mathcal{P}$ be given as in the introduction. We have
$$ \norm {\mathcal{Q}(u,v)}_{\tilde{L}^1(\dot{B}^{\frac{3}{p}-1}_{p,r})} \lesssim \norm u_{\mathscr{L}(\dot{B}^{\frac{3}{p}-1}_{p,r})} \norm v_{\mathscr{L}(\dot{B}^{\frac{3}{p}-1}_{p,r})},$$
$$ \norm {\mathcal{R}(u,v)}_{\tilde{L}^1(\dot{B}^{\frac{3}{p}-1}_{p,r})} \lesssim \norm u_{\mathscr{L}(\dot{B}^{\frac{3}{p}-1}_{p,r})} \norm v_{\mathscr{L}(\dot{B}^{\frac{3}{p}-1}_{p,r})},$$
$$ \norm {\mathcal{P}(u,v)}_{\tilde{L}^1(\dot{B}^{\frac{3}{p}}_{p,r})} \lesssim \norm u_{\mathscr{L}(\dot{B}^{\frac{3}{p}-1}_{p,r})} \norm v_{\mathscr{L}(\dot{B}^{\frac{3}{p}-1}_{p,r})}.$$
\end{prop}
\begin{proof}
 We will focus on the last inequality. The first two ones will follow by noticing that $\mathcal{Q}$ and $\mathcal{R}$ can be written as $\mathcal{D}^1(\mathcal{P})$, where $\mathcal{D}^1$ is a Fourier-multiplier of order 1.

 We consider the Bony's decomposition, to write
$$ uv = T_uv + T_vu+ R(u,v).$$
For the first term, we have
\begin{align*}
\norm{\Delta_j T_uv }_{L^1L^{p}_x}&\lesssim \sum_{k \sim j} \norm {S_{k-1} u}_{L^4 L^\infty_x} \norm {\Delta_k v}_{L^{\frac{4}{3}}L^p_x}.
\end{align*}
Using proposition \ref{prop.properties Besov}, we infer that
\begin{align*}
\norm{\Delta_j T_uv }_{L^1L^{p}_x}&\lesssim \sum_{k \sim j}c_{k,r}^2 2^{-k\frac{3}{p}} \norm { u}_{\tilde{L}^4 (\dot{B}^{-\frac{1}{2}}_{\infty,r})} \norm {v}_{\tilde{L}^\frac{4}{3}(\dot{B}^{\frac{3}{p}+\frac{1}{2}}_{p,r})}\\
&\lesssim  c_{j,r} 2^{-j\frac{3}{p}} \norm { u}_{\tilde{L}^4 (\dot{B}^{\frac{3}{p}-\frac{1}{2}}_{p,r})} \norm {v}_{\tilde{L}^\frac{4}{3}(\dot{B}^{\frac{3}{p}+\frac{1}{2}}_{p,r})}.
\end{align*} 
$T_vu$ enjoys the same estimate: by commuting $u$ and $v$ in the previous one, we obtain then
\begin{align*}
\norm{\Delta_j T_vu }_{L^1L^{p}_x}&\lesssim c_{j,r} 2^{-j\frac{3}{p}} \norm { v}_{\tilde{L}^4 (\dot{B}^{\frac{3}{p}-\frac{1}{2}}_{p,r})} \norm {u}_{\tilde{L}^\frac{4}{3}(\dot{B}^{\frac{3}{p}+\frac{1}{2}}_{p,r})}.
\end{align*}
Finally, the remainder term, can be dealt with along the same lines, we infer that
\begin{align*}
\norm {\Delta_j R(u,v)}_{L^1(L^p_x)} &\lesssim   \sum _{k\geq j+N_0} \norm {\tilde{\Delta}_k u}_{L^4(L^\infty_x)}\norm {\Delta_k v}_{L^\frac{4}{3}(L^{p }_x)}\\
&\lesssim 2^{-j \frac{3}{p}} \sum _{k\geq j+N_0} 2^{(k-j) \frac{3}{p}} c_{k,r}^2 \norm { u}_{\tilde{L}^4 (\dot{B}^{-\frac{1}{2}}_{\infty,r})} \norm {v}_{\tilde{L}^\frac{4}{3}(\dot{B}^{\frac{3}{p}+\frac{1}{2}}_{p,r})}\\
&\lesssim  c_{j,r} 2^{-j\frac{3}{p}} \norm { u}_{\tilde{L}^4 (\dot{B}^{\frac{3}{p}-\frac{1}{2}}_{p,r})} \norm {v}_{\tilde{L}^\frac{4}{3}(\dot{B}^{\frac{3}{p}+\frac{1}{2}}_{p,r})}.
\end{align*}
Thanks to the embedding $\mathscr{L}(\dot{B}^{\frac{3}{p}}_{p,r}) \hookrightarrow \tilde{L}^\frac{4}{3}(\dot{B}^{\frac{3}{p}+\frac{1}{2}}_{p,r}) \cap \tilde{L}^4 (\dot{B}^{\frac{3}{p}-\frac{1}{2}}_{p,r}),$ 
then we obtain
\begin{equation}\label{product uv}
\norm {uv}_{\tilde{L}^1(\dot{B}^{\frac{3}{p}}_{p,r})} \lesssim \norm u_{\mathscr{L}(\dot{B}^{\frac{3}{p}-1}_{p,r})} \norm v_{\mathscr{L}(B^{\frac{3}{p}-1}_{p,r})}
\end{equation}
It follows that
\begin{equation}
\norm {\mathcal{P}(u,v)}_{\tilde{L}^1(\dot{B}^{\frac{3}{p}}_{p,r})}\approx\norm {uv}_{\tilde{L}^1(\dot{B}^{\frac{3}{p}}_{p,r})} \lesssim \norm u_{\mathscr{L}(\dot{B}^{\frac{3}{p}-1}_{p,r})} \norm v_{\mathscr{L}(\dot{B}^{\frac{3}{p}-1}_{p,r})},
\end{equation}
and 
\begin{equation}
\norm {\mathcal{Q}(u,v)}_{\tilde{L}^1(\dot{B}^{\frac{3}{p}-1}_{p,r})} \approx \norm {\mathcal{R}(u,v)}_{\tilde{L}^1(\dot{B}^{\frac{3}{p}-1}_{p,r})} \lesssim \norm u_{\mathscr{L}(\dot{B}^{\frac{3}{p}-1}_{p,r})} \norm v_{\mathscr{L}(\dot{B}^{\frac{3}{p}-1}_{p,r})}.
\end{equation}
Proposition \ref{prop.continuity} is proved.
\end{proof}
\begin{rmq}
We can replace $\tilde{L}^1(\cdot)$ and $  \mathscr{L} (\cdot)$ in the previous proposition, respectively, by $\tilde{L}^1_T(\cdot) $ and $  \mathscr{L }_T(\cdot)$, for $T>0.$
\end{rmq}
 In order to prove Theorem \ref{thm1}, we will use the following the abstract Banach fixed point theorem stated in the following lemma. The reader can see Lemma 4 in \cite{handbook} for more details	
 \begin{lemma}\label{lemma fixed point}
Let $\mathcal{X}$ be an abstract Banach space with norm $\norm
 .$, let $\zeta$ be a bi-linear operator mapping $\mathcal{X} \times \mathcal{X} $ into $\mathcal{X}$ satisfying
 $$\forall x_1,x_2 \in \mathcal{X} ,\;\; \norm {\zeta(x_1,x_2)}\leq \eta \norm {x_1} \norm {x_2}, \;\; \text{ for some } \eta >0,$$
 then for all $y\in \mathcal{X}$ such that 
 $$ \norm y < \frac{1}{4\eta}$$
 the equation $$x=y+\zeta(x,x)$$
 has a solution $x\in \mathcal{B}_{\mathcal{X}}\big (0,2\norm y\big)$.\\
 This solution is the unique one in the ball $\mathcal{B}_{\mathcal{X}}\big (0,\frac{1}{2 \eta}\big) $
\end{lemma}
\begin{proof}[\textbf{\textit{Proof of Theorem \ref{thm1}}}]
In order to apply Lemma \ref{lemma fixed point}, all we need to show is that 
$$ \norm { \zeta(\mathcal{U}, \mathcal{V})}_{\mathscr{L}(\dot{B}^{\frac{3}{p}-1}_{p,r} )} \lesssim \norm {\mathcal{U}}_{\mathscr{L}(\dot{B}^{\frac{3}{p}-1}_{p,r})} \norm{\mathcal{V}}_{\mathscr{L}(\dot{B}^{\frac{3}{p}-1}_{p,r} )}$$
$$ \norm { e^{t\Delta}\mathcal{U}_0 }_{\mathscr{L}(\dot{B}^{\frac{3}{p}-1}_{p,r} )} \lesssim \norm {\mathcal{U}_0}_{\dot{B}^{\frac{3}{p}-1}_{p,r}}.$$
The second inequality follows directly from inequality \eqref{heat kernel 1} from Lemma \ref{lemma heat kernel} in Appendix, and the first one follows by combining Proposition \ref{prop.continuity}, Proposition \ref{prop.laplacian} and inequality \eqref{heat kernel 2} from Lemma \ref{lemma heat kernel}. Indeed,
 $\zeta = (\zeta_1,\zeta_2,\zeta_3)^T$ contains in each component the bi-linear operators $\mathcal{Q}, \mathcal{P}, \mathcal{R}$:
\begin{itemize}
\item  For $\mathcal{Q}$ and $\mathcal{R}$: 
\begin{itemize}
\item in $\zeta_1$ we apply directly Proposition \ref{prop.continuity},
\item in $\zeta_2$ we apply Proposition \ref{prop.continuity} and inequality \eqref{inequality 1.prop laplace} for $\rho=1$ and $k= 1$
\item in $\zeta_2$ we apply Proposition \ref{prop.continuity} and inequality \eqref{inequality 2.prop laplace} for $\rho=1$ and $k= 2$
\end{itemize}
\item For $\mathcal{P}$
\begin{itemize}
\item in $\zeta_2$ we apply Proposition \ref{prop.continuity} and inequality \eqref{inequality 1.prop laplace} for $\rho=1$ and $k= 0$
\item in $\zeta_2$ we apply Proposition \ref{prop.continuity} and inequality \eqref{inequality 2.prop laplace} for $\rho=1$ and $k= 1$
\end{itemize}
\end{itemize} 
This ends the proof of Theorem \ref{thm1}.
\end{proof}
As a corollary, one may replace the smallness condition on the initial data in Theorem \ref{thm1}, by another one on the maximal time of existence, namely we can prove:
\begin{cor}\label{corollary local existence}
Let $p\in [1, \infty), \; r\in [1,\infty]$, and $\mathcal{U}_0=(u_0,B_0, \nabla \times B_0)$, be in  $\dot{B}^{\frac{3}{p}-1}_{p,r}(\mathbb{R}^3) $.  \\
There exists $T^*>0$ and a unique solution $\mathcal{U}$ to \eqref{S zeta} in $ \mathscr{L}_T (\dot{B}^{\frac{3}{p}-1}_{p,r} )$, for all $T<T^*$.
\end{cor}
\begin{proof} 
We split the solution $\mathcal{U}$ into a sum $$\mathcal{U} = \mathcal{V}+ \mathcal{W},$$ where $\mathcal{V}$ is given by
$$ \mathcal{V}(t,\cdot) \bydef e^{t\Delta} \mathcal{U}_{0}.$$
then it remains to solve, by fixed point argument, the equation on $\mathcal{W}$
$$ \mathcal{W} =  2\zeta(\mathcal{V},\mathcal{V})+2 \zeta(\mathcal{W},\mathcal{W}) +  \zeta(\mathcal{V},\mathcal{W})+  \zeta(\mathcal{W},\mathcal{V}).$$
Let us point out that, according to the previous calculations, we have
$$\norm { \Omega(\mathcal{W}, \mathcal{W})}_{\tilde{L}^1_T(\dot{B}^{\frac{3}{p}-1}_{p,r})} \leq \gamma \norm {\mathcal{W}}_{\mathscr{L}_T(\dot{B}^{\frac{3}{p}-1}_{p,r})}^2 $$
for some $\gamma>0$, where $\Omega$ is given by \eqref{omega expression}. We also have
$$ \norm {\zeta(\mathcal{V},\mathcal{V})}_{\mathscr{L}_T(\dot{B}^{\frac{3}{p}-1}_{p,r})} =\norm { \mathcal{K} \Omega(\mathcal{V}, \mathcal{V}) }_{\mathscr{L}_T(\dot{B}^{\frac{3}{p}-1}_{p,r})}\lesssim \norm { \Omega(\mathcal{V}, \mathcal{V})}_{\tilde{L}^1_T(\dot{B}^{\frac{3}{p}-1}_{p,r})}.$$
On the other hand we know that $\Omega(\mathcal{V}, \mathcal{V}) \in \tilde{L}^1_T(\dot{B}^{\frac{3}{p}-1}_{p,r})$ from the estimates of Proposition \ref{prop.continuity}, and Lemma \ref{lemma heat kernel}, which gives in particular
$$\norm { \Omega(\mathcal{V}, \mathcal{V})}_{\tilde{L}^1_T(\dot{B}^{\frac{3}{p}-1}_{p,r})} \lesssim \norm {\mathcal{V}}_{\mathscr{L}_T(\dot{B}^{\frac{3}{p}-1}_{p,r})}^2 \lesssim \norm {\mathcal{V}_0}_{\dot{B}^{\frac{3}{p}-1}_{p,r} }^2<\infty.$$
For the linear term on $\mathcal{W}$, by virtue of Lemma \ref{lemma heat kernel} and the proof of Proposition \ref{prop.continuity}, where we showed that we can obtain the estimates of $\mathcal{Q},$ $ \mathcal{P}$ and $\mathcal{R}$, by using only the norm of $V$ in $   \mathcal{Z}_T\bydef\tilde{L}^4_T(\dot{B}^{\frac{3}{p}-\frac{1}{2}})\cap \tilde{L}^\frac{4}{3}_T(\dot{B}^{\frac{3}{p}+\frac{1}{2}}) $, we infer that
$$\norm {\zeta(\mathcal{V},\mathcal{W})}_{\mathscr{L}_T(\dot{B}^{\frac{3}{p}-1}_{p,r})}+ \norm {\zeta(\mathcal{W},\mathcal{V})}_{\mathscr{L}_T(\dot{B}^{\frac{3}{p}-1}_{p,r})} \leq C \norm {\mathcal{V}}_{\mathcal{Z}_T} \norm { \mathcal{W}}_{\mathscr{L}_T(\dot{B}^{\frac{3}{p}-1}_{p,r})},$$
for some universal constant $C>0$.
We chose $T_1$ such that, for all $T<T_1$
$$C\norm {\mathcal{V}}_{\mathcal{Z}_T} <  \lambda<1,$$
we can then chose $T^*\leq T_1$ small enough so that, for all $T<T^*\leq T_1$, we have
$$\norm { \Omega(\mathcal{V}, \mathcal{V})}_{\tilde{L}^1_T(\dot{B}^{\frac{3}{p}-1}_{p,r})} < \varepsilon_1 \leq \frac{(1-\lambda)^2}{4\gamma}$$
The result follows by a direct application of Lemma \ref{lemma fixed point3}.
\end{proof}
\subsection{Reduction of the required initial regularity on the magnetic field: Proof of Theorem \ref{thm2}}
The key estimates to prove Theorem \ref{thm2} is shown in the following proposition
\begin{prop}\label{prop.continuity2}
Let $p$ be in $[1, \infty)$, $T>0$, $u$ in $\mathscr{L}_T(\dot{B}^{\frac{3}{p}-1}_{p,1})$ and $w,z$ be in $\mathscr{L}_T(\dot{B}^{\frac{3}{p}}_{p,1})$, then we have
\begin{equation}\label{inequality 1 prop continuity 2}
 \norm {\mathcal{Q}(w,z)}_{\tilde{L}^1_T(\dot{B}^{\frac{3}{p}-1}_{p,1})} \lesssim T\norm w_{\mathscr{L}_T(\dot{B}^{\frac{3}{p} }_{p,1})} \norm z_{\mathscr{L}_T(B^{\frac{3}{p} }_{p,1})}
\end{equation} 
\begin{equation}\label{inequality 0 prop continuity 2}
\norm {\mathcal{Q}(w,z)}_{\tilde{L}^1_T(\dot{B}^{\frac{3}{p}+1}_{p,1})} \lesssim  \norm w_{\mathscr{L}_T(\dot{B}^{\frac{3}{p} }_{p,1})} \norm z_{\mathscr{L}_T(B^{\frac{3}{p} }_{p,1})}
\end{equation}
 \begin{equation}\label{inequality 2 prop continuity 2}
 \norm {\mathcal{P}(u,w)}_{\tilde{L}^1_T(\dot{B}^{\frac{3}{p}+1}_{p,1})} \lesssim \norm u_{\mathscr{L}_T(\dot{B}^{\frac{3}{p}-1}_{p,1})} \norm w_{\mathscr{L}_T(B^{\frac{3}{p}}_{p,1})} 
 \end{equation}
 \begin{equation}\label{inequality 3 prop continuity 2}
 \norm {\mathcal{R}(u,w)}_{\tilde{L}^1_T(\dot{B}^{\frac{3}{p} }_{p,1})} \lesssim \norm u_{\mathscr{L}_T(\dot{B}^{\frac{3}{p}-1}_{p,1})} \norm w_{\mathscr{L}_T(B^{\frac{3}{p}}_{p,1})} 
 \end{equation}
\end{prop}
The proof does not work for $r>1$, as the embedding $\dot{B}^{\frac{3}{p}}_{p,r}(\mathbb{R}^3)\hookrightarrow L^\infty(\mathbb{R}^3) $ fails to be true unless when $r=1$. In this part of the thesis, we will denote $d_j\bydef c_{j,1}$ being a sequence in $\ell^1(\mathbb{Z})$.\\
\begin{proof}[Proof of Proposition \ref{prop.continuity2}]~~\\
\textit{Proof of} \eqref{inequality 1 prop continuity 2}:
Inequality \eqref{inequality 1 prop continuity 2} follows directly from the fact that $\tilde{L}^\infty(\dot{B}^{\frac{3}{p}}_{p,1})$ is an algebra and the (local in time) embedding
$$ \norm a_{\tilde{L}^1(\dot{B}^{\frac{3}{p}}_{p,1})} \leq T\norm a_{\tilde{L}^\infty(\dot{B}^{\frac{3}{p}}_{p,1})}. $$
\textit{Proof of} \eqref{inequality 0 prop continuity 2}:\\
According to Bony's decomposition, we have
$$wz = T_w z + T_zw + R(w,z). $$
We then show how to estimate the first and the third term. We have
\begin{align*}
\norm {\Delta_jT_wz}_{L^1_TL^p}&\lesssim \norm {S_{j-1}w}_{L^\infty_T L^\infty} \norm {\Delta_j z}_{L^1_T L^p}\\
&\lesssim d_j 2^{-j(\frac{3}{p}+2)}\norm w_{L^\infty_T(L^\infty)} \norm w_{\tilde{L}^1_T(\dot{B}^{\frac{3}{p}+2}_{p,1})}.
\end{align*}
We deduce from the embedding $ \dot{B}^\frac{3}{p}_{p,1} (\mathbb{R}^3)\hookrightarrow L^\infty(\mathbb{R}^3)$, together with Minkowski's inequality, give
$$\norm {\Delta_jT_wz}_{L^1_TL^p} \lesssim d_j 2^{-j(\frac{3}{p}+2)}\norm w_{\tilde{L}^\infty_T(\dot{B}^\frac{3}{p}_{p,1})} \norm w_{\tilde{L}^1_T(\dot{B}^{\frac{3}{p}+2}_{p,1})}. $$
For the remainder term, we proceed as follows
\begin{align*}
\norm {\Delta_j R(w,z)}_{L^1_TL^p}&\lesssim \sum _{k\geq j+N_0} \norm {\tilde{\Delta}_k w}_{L^1_TL^p} \norm {\Delta_k z}_{L^\infty L^\infty}\\
&\lesssim 2^{ -j(\frac{3}{p}+2)} \sum _{k\geq j+N_0} d_k 2^{(j-k)(\frac{3}{p}+2)} \norm {  w}_{\tilde{L}^1_T(\dot{B}^{\frac{3}{p}+2}_{p,1})} \norm { z}_{\tilde{L}^\infty (\dot{B}^0_{\infty,\infty}) }\\
&\lesssim 2^{ -j(\frac{3}{p}+2)} d_j \norm { w}_{\tilde{L}^1_T(\dot{B}^{\frac{3}{p}+2}_{p,1})} \norm { z}_{\tilde{L}^\infty (\dot{B}^{\frac{3}{p}}_{p,1}) }.
\end{align*}
Inequality \eqref{inequality 0 prop continuity 2} follows.\\
 \textit{proof of} \eqref{inequality 2 prop continuity 2}: 
Let us point out again that, due to the fact that $\mathcal{D}^1( \mathcal{P}) \approx \mathcal{R} $, \eqref{inequality 2 prop continuity 2} and \eqref{inequality 3 prop continuity 2} can be proved along the same way, we will thus concentrate on the proof of \eqref{inequality 2 prop continuity 2}.\\
We consider again the Bony's decomposition
$$ uw = T_uw + T_wu + R(u,w).$$
For $T_uw $, we have
\begin{align*}
\norm {\Delta_j (T_uw)}_{L^1_TL^p} &\lesssim \norm {S_{j-1} u}_{L^\infty_T L^\infty} \norm {\Delta_jw}_{L^1_T L^p}\\
&\lesssim d_j 2^{-j(\frac{3}{p}+1)}\norm u_{\tilde{L}^\infty_T(\dot{B}^{-1}_{\infty,\infty})} \norm w_{\tilde{L}^1_T(\dot{B}^{\frac{3}{p}+2}_{p,1})}\\
&\lesssim d_j 2^{-j(\frac{3}{p}+1)}\norm u_{ \tilde{L}^\infty_T(\dot{B}^{\frac{3}{p}-1}_{p,1})} \norm w_{\tilde{L}^1_T(\dot{B}^{\frac{3}{p}+2}_{p,1})}.
\end{align*}
For $T_wu$, by using the embedding $\dot{B}^{\frac{3}{p}}_{p,1}(\mathbb{R}^3)\hookrightarrow L^\infty(\mathbb{R}^3)$, we infer that
\begin{align*}
\norm {\Delta_j (T_wu)}_{L^1_TL^p} &\lesssim \norm {S_{j-1} w}_{L^\infty_T L^\infty} \norm {\Delta_ju}_{L^1_T L^p}\\
&\lesssim d_j 2^{-j(\frac{3}{p}+1)}\norm w_{L^\infty_T(L^\infty)} \norm u_{\tilde{L}^1_T(\dot{B}^{\frac{3}{p}+1}_{p,1})}\\
&\lesssim d_j 2^{-j(\frac{3}{p}+1)}\norm w_{L^\infty_T(\dot{B}^{\frac{3}{p}}_{p,1} )} \norm u_{\tilde{L}^1_T(\dot{B}^{\frac{3}{p}+1}_{p,1})}
\end{align*}
Minkoski's inequality then gives  
$$ \norm {\Delta_j (T_wu)}_{L^1_TL^p} \lesssim d_j 2^{-j(\frac{3}{p}+1)}\norm w_{\tilde{L}^\infty_T(\dot{B}^{\frac{3}{p}}_{p,1} )} \norm u_{\tilde{L}^1_T(\dot{B}^{\frac{3}{p}+1}_{p,1})}.$$
For the remainder term, we have
\begin{align*}
\norm {\Delta_j R(u,w)}_{L^1_TL^p}&\lesssim \sum _{k\geq j+N_0} \norm {\tilde{\Delta}_k u}_{L^1_TL^p} \norm {\Delta_k w}_{L^\infty_T L^\infty}\\
&\lesssim 2^{ -j(\frac{3}{p}+1)} \sum _{k\geq j+N_0} d_k 2^{(j-k)(\frac{3}{p}+1)} \norm {  u}_{\tilde{L}^1_T(\dot{B}^{\frac{3}{p}+1}_{p,1})} \norm { w}_{\tilde{L}^\infty_T (\dot{B}^0_{\infty,\infty}) }\\
&\lesssim 2^{ -j(\frac{3}{p}+1)} d_j \norm {  u}_{\tilde{L}^1_T(\dot{B}^{\frac{3}{p}+1}_{p,1})} \norm { w}_{\tilde{L}^\infty_T (\dot{B}^{\frac{3}{p}}_{p,1}) }.
\end{align*}
This ends the proof of inequality \eqref{inequality 2 prop continuity 2}, and eventually \eqref{inequality 3 prop continuity 2}. \\
Lemma \ref{prop.continuity2} is then proved.
\end{proof}
\noindent The proof of Theorem \ref{thm2} is based on the following variation of Lemma \ref{lemma fixed point}.
\begin{lemma}\label{lemma fixed point2}
Let $\{A_i\}_{i\in \{1,2,3,4\}}$ be a set of bi-linear operators with
$$ \norm {A_1(x_1,x_2)}_{\mathcal{X}} \leq \eta_1 \norm {x_1}_{\mathcal{X}}\norm {x_2}_{\mathcal{X}}$$
$$\norm {A_2(y_1,y_2)}_{\mathcal{X}} \leq (1+T)\eta_2 \norm {y_1}_{\mathcal{Y}}\norm {y_2}_{\mathcal{Y}} $$
$$ \norm {A_3(x_1,y_2)}_{\mathcal{Y}} \leq \eta_3 \norm {x_1}_{\mathcal{X}}\norm {y_2}_{\mathcal{Y}}$$
$$ \norm {A_4(y_1,y_2)}_{\mathcal{Y}} \leq \eta_4\norm {y_1}_{\mathcal{Y}}\norm {y_2}_{\mathcal{Y}} $$
for some non negative $T,(\eta_i)_{i\in \{1,2,3,4\}}$.\\
Let $\eta \bydef \underset{i\in \{1,2,3,4\}}{\max\{ \eta_i \}} $. Then for all $(x_0,y_0)\in(\mathcal{X} \times \mathcal{Y})$ such that
\begin{equation}\label{hyp fixed point 2}
 \norm {x_0}_{\mathcal{X}}+  (2+T)  \norm {y_0}_{\mathcal{Y}}    < \frac{1}{24 \eta} ,
\end{equation}
the system 
\begin{center}
$ \left\{\begin{array}{l}
  x= x_0 + A_1(x,x)+ A_2(y,y)\\
  y=y_0 + A_3(x,y) + A_4(y,y)
\end{array}\right.$
\end{center}
has a unique solution $(x,y)$ in $ \mathcal{X} \times \mathcal{Y} $, which also satisfies
$$ \norm {x}_{\mathcal{X}}+  (2+T)  \norm {y}_{\mathcal{Y}}    < \frac{1}{12 \eta} . $$
\end{lemma}
\begin{proof}
Let us present briefly the outlines of the proof, the idea is classical:\\
We define the sequence $(x^n,y^n )$ by
\begin{center}
$ \left\{\begin{array}{l}
  (x^0,y^0) = (x_0,y_0)\\
  x^{n+1}= x_0 + A_1(x^n,x^n)+  A_2(y^n,z^n)\\
  y^{n+1}=y_0 + A_3(x^n,y^n) + A_4(y^n,y^n)\\
\end{array}\right.$
\end{center}
If we write $z^n \bydef (1+T) y^n$, then the system above is equivalent to
\begin{center}
$ (Seq )\left\{\begin{array}{l}
  (x^0,y^0,z^0) = (x_0,y_0,(1+T) y^0)\\
  x^{n+1}= x_0 + A_1(x^n,x^n)+ \tilde{A}_2(y^n,z^n)\\
  y^{n+1}=y_0 + A_3(x^n,y^n) + A_4(y^n,y^n)\\
    z^{n+1}=z_0 + A_3(x^n,z^n) + A_4(y^n,z^n)
\end{array}\right.$
\end{center}
with $\tilde{A}_2 = \frac{1}{1+T} A_2$ whose norm is less than $\eta$.\\
Let $\alpha\bydef \norm {x_0}_{\mathcal{X}} + \norm {y_0}_{\mathcal{Y}} + \norm {z_0}_{\mathcal{Y}} < \frac{1}{24\eta }$, we claim that $(x^n,y^n,z^n)$ is a Cauchy (bounded) sequence in $\mathcal{B}_{\mathcal{X} \times \mathcal{Y}\times \mathcal{Y} }\big (0, 2\alpha )$.\\
By virtue of the definition of $(x^n,y^n,z^n)$ and the continuity of $A_i$, we proceed by induction to obtain 
\begin{center}
$   \left\{\begin{array}{l}
  \norm {x^{n+1}}_{\mathcal{X}} \leq \norm {x_0}_{\mathcal{X}} +      8\eta \alpha^2 \\
  \norm {y^{n+1}}_{\mathcal{Y}} \leq \norm {y_0}_{\mathcal{X}} +    8 \eta \alpha^2 \\
  \norm {y^{z+1}}_{\mathcal{Y}} \leq \norm {z_0}_{\mathcal{X}} +     8\eta \alpha^2  
\end{array}\right.$
\end{center}
which gives 
$$ \norm {x^{n+1}}_{\mathcal{X}} + \norm {y^{n+1}}_{\mathcal{Y}} + \norm {z^{n+1}}_{\mathcal{Y}} < 2\alpha. $$
In order to prove that $(x^n,y^n,z^n)$ is a Cauchy sequence, similar computation lead to
$$ I_n \bydef  \norm {x^{n+1} - x^n}_{\mathcal{X}} + \norm {y^{n+1}- y^n}_{\mathcal{Y}} + \norm {z^{n+1}-z^n}_{\mathcal{Y}} \leq   (24 \eta\alpha)  I_{n-1}      
  $$ 
  This is enough to conclude the proof.
\end{proof}
\begin{proof}[\textbf{\textit{Proof of Theorem \ref{thm2}}}]
In order to apply Lemma \ref{lemma fixed point2}, let us rewrite the system \eqref{S zeta} as follows. We define
 $$\begin{pmatrix}
   \psi_1(x_1,x_2) \\
   \psi _2(y_1,y_2)\\
    \psi_3 (x_1,y_1)\\
   \psi_4(y_1,y_2)
 \end{pmatrix} \bydef \begin{pmatrix}
 - \mathcal{Q}( x_1,x_2  ) \\
 \mathcal{Q}( y_1, y_2  )- \mathcal{Q}(\nabla \times y_1,\nabla \times y_2)   \\
 \mathcal{P}(x_1,y_1) - \mathcal{R}( \nabla \times y_1, x_1) - 2 \mathcal{Q}(\nabla \times y_1, x_1)\\
  -2 \mathcal{Q}( y_1,y_2) - \mathcal{Q}( \nabla \times y_1, \nabla \times y_2) 
 \end{pmatrix},
 $$
 so that
 $$\varphi \begin{pmatrix}
 x_1 & x_2\\
 y_1&y_2
\end{pmatrix}  = \begin{pmatrix}
 \varphi  _1 (x_1,x_1)\\
 \varphi  _2 (y_1,y_2)\\
  \varphi  _3 (x_1,y_1)\\
  \varphi_4(y_1,y_2)
 \end{pmatrix} \bydef \begin{pmatrix}
 \psi_1(x_1,x_2) \\
   \psi _2(y_1,y_2)\\
   (Id-\Delta)^{-1} \nabla \times \psi_3 (x_1,y_1 ) \\
   (Id-\Delta)^{-1} \nabla \times \psi_4 (y_1,y_2 )  
 \end{pmatrix},
 $$
 and then, we define
  $$\begin{pmatrix}
    A_1(x_1,x_2) \\
    A _2(y_1,y_2)\\
    A_3 (x_1,y_1)\\
   A_4(y_1,y_2)
 \end{pmatrix} \bydef  \mathcal{K} \mathbb{P}\varphi\begin{pmatrix}
 x_1 & x_2\\
 y_1&y_2
\end{pmatrix},
 $$
 with $$ \mathcal{K}\mathbb{P}\varphi (t,\cdot) = \int_0^t e^{(t-s)\Delta}\mathbb{P} \varphi(s,\cdot) ds.$$
 Therefore, the system \eqref{S zeta} is equivalent to the following one
 \begin{center}
$ \left\{\begin{array}{l}
  u(t, \cdot)= e^{t\Delta}u_0 + A_1(u,u)+ A_2(B,B)\\
  B(t,\cdot)=e^{t\Delta}B_0 + A_3(u,B) + A_4(B,B)
\end{array}\right.$
\end{center}
The proof of Theorem \ref{thm2} can be reduced to a direct application of Lemma \ref{lemma fixed point2}, thus all we need to show then is that $\{A_i\}_{i\in\{1,2,3,4\}}$ satisfies the hypothesis of Lemma \ref{lemma fixed point2} with $\mathcal{X} = \mathscr{L}(\dot{B}^{\frac{3}{p}-1}_{p,1})$ and $\mathcal{Y}= \mathscr{L}(\dot{B}^{\frac{3}{p} }_{p,1})$. To do so, according to Lemma \ref{lemma heat kernel} we should estimate $ \varphi$ in $$\tilde{L}^1_T \big( \dot{B}^{\frac{3}{p}-1}_{p,1} \times \dot{B}^{\frac{3}{p}-1}_{p,1}\times \dot{B}^{\frac{3}{p} }_{p,1} \times \dot{B}^{\frac{3}{p} }_{p,1} \big).$$
Now, each component of $\varphi$ contains a combination of $\mathcal{Q}$, $ \mathcal{P},$ and $\mathcal{R}$, we will thus show how to use Proposition \ref{prop.continuity}, Proposition \ref{prop.continuity2} and Lemma \ref{lemma multiplicateur fourier} to deal with each one
\begin{itemize}
\item For $\mathcal{Q}$
\begin{itemize}
\item in $\varphi_1$, we apply Proposition \ref{prop.continuity}.
\item in $ \varphi_2 $, for $\mathcal{Q}(y_1,y_2)$ we apply inequality \eqref{inequality 1 prop continuity 2} from Proposition \ref{prop.continuity2}, and for \\ $ \mathcal{Q}(\nabla \times y_1, \nabla \times y_2)$,  we apply Proposition \ref{prop.continuity}.
\item in $\varphi_3$, we apply respectively Proposition \ref{prop.continuity}, then inequality \eqref{inequality 1.prop laplace} from Proposition \ref{prop.laplacian} for $k= 2$.
\item in $\varphi_4$,
\begin{itemize}
\item for $\mathcal{Q}(y_1,y_2)$, we apply inequality \eqref{inequality 0 prop continuity 2} from Proposition \ref{prop.continuity2}, then inequality \eqref{inequality 1.prop laplace} from Proposition \ref{prop.laplacian} for $k= 0$
\item for $\mathcal{Q}(\nabla \times y_1, \nabla \times y_2)$, we apply Proposition \ref{prop.continuity}, then inequality \eqref{inequality 1.prop laplace} from Proposition \ref{prop.laplacian} for $k= 2$.
\end{itemize}
\end{itemize}
\item For $\mathcal{P} $, we apply inequality \eqref{inequality 2 prop continuity 2} from Proposition \ref{prop.continuity2}, then inequality \eqref{inequality 1.prop laplace} from Proposition \ref{prop.laplacian} for $k= 0$.
\item For $\mathcal{R} $, we apply inequality \eqref{inequality 3 prop continuity 2} from Proposition \ref{prop.continuity2}, then inequality \eqref{inequality 1.prop laplace} from Proposition \ref{prop.laplacian} for $k= 1$.
\end{itemize}
\end{proof}
This ends the proof of Theorem \ref{thm2}.
\subsection{The wellposedness in Kato-Herz spaces and the long time behavior: Proof of Theorem \ref{thm3}}
In this section, we shall give some details about the wellposedeness in the hat-Kato space $\widehat{K}^{1-\frac{3}{p}}_{p}$, and then we establish the decay property described in Theorem \ref{thm3}. It is all based on the following proposition
\begin{prop}\label{prop.continuity kato}
Let $\zeta$ be given by \eqref{zeta definition}. Then $\zeta$ maps $\widehat{K}^{1-\frac{3}{p}}_p \times \widehat{K}^{1-\frac{3}{p}}_p$ continuously into $\widehat{K}^{1-\frac{3}{p}}_p$. That is to say, there exists $\kappa>0$ such that
$$ \norm { \zeta (L,M)}_{\widehat{K}^{1-\frac{3}{p}}_p} \leq \kappa \norm L_{\widehat{K}^{1-\frac{3}{p}}_p} \norm M_{\widehat{K}^{1-\frac{3}{p}}_p}, $$
for all $L,M\in \widehat{K}^{1-\frac{3}{p}}_p$.
\end{prop}
\begin{proof}
By taking into account the inequality, for all $m\in [0,2] $ and $|\xi| \sim 2^j$
 \begin{equation}\label{holder inequality}
 \frac{1}{1+|\xi|^2} \lesssim \frac{1}{2^{jm}}  ,
 \end{equation}
in order to prove the continuity property of $ \zeta$ on $\widehat{K}^{1-\frac{3}{p}}_p$, we only need to show that, for all $t>0$
\begin{equation}\label{inequality continuation zeta in kato hat}
t^{\frac{1}{2}(1-\frac{3}{p})}  \int_0^t\norm { e^{-(t-s) |\cdot|^2} |\cdot| (\widehat{L} \star \widehat{M})(s,\cdot) }_{L^{p'}}  ds \lesssim \norm L_{\widehat{K}^{1-\frac{3}{p}}_p} \norm M_{\widehat{K}^{1-\frac{3}{p}}_p},
\end{equation}
that is, due to \eqref{holder inequality}, all the components of $\zeta$ can be dominated by a Gaussian multiplied by an order one Fourier-multiplier, as in the proof of Theorem \ref{thm2}. 

Let us then give  some details about the proof of \eqref{inequality continuation zeta in kato hat}. By setting, for $p>3$,
$$ \frac{1}{r}\bydef 1- \frac{2}{p }\iff \frac{1}{p}+ \frac{1}{r}= \frac{1}{p'} \iff \frac{1}{p'}+ \frac{1}{p'}- \frac{1}{r}=1,$$
 by virtue of Holder inequality, we infer that
 \begin{align*}
 \int_0^t\norm { e^{-(t-s) |\cdot|^2} |\cdot| (\widehat{L} \star \widehat{M})(s,\cdot) }_{L^{p'}}  ds \lesssim \int_0^t\norm {\mathcal{G}(t-s,\cdot)}_{L^p} \norm {\widehat{L}(s,\cdot)}_{L^{p'}}\norm {\widehat{M}(s,\cdot)}_{L^{p'}}ds,
 \end{align*}
 where $$\mathcal{G}(\tau,\xi) \bydef e^{-\tau|\xi|^2} |\xi|.$$
 By a change of variable in the $L^p$ norm of $\mathcal{G}(t-s,\cdot)$, we obtain
 $$\norm {\mathcal{G}(t-s,\cdot)}_{L^p} \lesssim  \frac{1}{(t-s)^{\frac{1}{2}+ \frac{3}{2p}}}\norm {\mathcal{G}(1,\cdot)}_{L^p}.$$
 This gives
 \begin{align*}
  \int_0^t\norm { e^{-(t-s) |\cdot|^2} |\cdot| (\widehat{L} \star \widehat{M})(s,\cdot) }_{L^{p'}}  ds  &\lesssim \norm L_{\widehat{K}^{1-\frac{3}{p}}_p} \norm M_{\widehat{K}^{1-\frac{3}{p}}_p} \int_0^t\frac{1}{(t-s)^{\frac{1}{2}+ \frac{3}{2p}}  s^{1-\frac{3}{p}}}  ds,\\
  &\lesssim  t^{-\frac{1}{2}(1-\frac{3}{p})}\norm L_{\widehat{K}^{1-\frac{3}{p}}_p} \norm M_{\widehat{K}^{1-\frac{3}{p}}_p}.
 \end{align*}
 Inequality \eqref{inequality continuation zeta in kato hat} follows, and Proposition \ref{prop.continuity kato} is then proved.
\end{proof}
\begin{proof}[\textbf{\textit{Proof of theorem \ref{thm3}}}]
According to proposition \ref{prop.continuity kato}, we can apply the fixed point lemma \ref{lemma fixed point}, that is, if $$ \norm {e^{t\Delta}u_0}_{\widehat{K}_p^{\frac{3}{p}-1}}+\norm {e^{t\Delta}B_0}_{\widehat{K}_p^{\frac{3}{p}-1}}+\norm {e^{t\Delta}(\nabla \times B_0)}_{\widehat{K}_p^{\frac{3}{p}-1}}<\frac{1}{4\kappa}$$
then we can construct a unique solution $\mathcal{U}$ of \eqref{S zeta} in $\widehat{K}^{\frac{3}{p}-1}_p$ with $$\norm {\mathcal{U}}_{\widehat{K}^{\frac{3}{p}-1}_p}< \frac{1}{2\kappa}.$$\\
By virtue of the continuity of the Fourier transform, from $L^{q}$ into $L^{q'}$, for $q\in [1,2]$, we infer that, for $p \in (3,\infty)$,
$$\norm {\mathcal{U}(t,\cdot)}_{L^p}=\norm {\widehat{\widehat{\mathcal{U}}}(t,\cdot)}_{L^p}   \lesssim \norm {\widehat{\mathcal{U}}(t,\cdot)}_{L^{p'}}$$
it follows then, for all $t>0$
$$ \norm {\mathcal{U}(t,\cdot)}_{L^p} \lesssim \norm {\widehat{\mathcal{U}}(t,\cdot)}_{L^{p'}} \lesssim t^{-\frac{1}{2}(1-\frac{3}{p})}.$$
Theorem \ref{thm3} is then proved.
\end{proof}
\appendix
\section{~~}
In this section we recall some basic tools of a constant use in the analysis of our paper. Some useful results are also proved in the sequel.\\
 
\noindent We begin by recalling the Bernstein lemma from \cite{Chemin1}
\begin{lemmasec}[Bernstein] \label{bernstein}
Let $\mathcal{B} $ be a ball of $\mathbb{R}^d$, and $\mathcal{C} $  be a ring of $\mathbb{R}^d$. Let also $a$ be a tempered distribution and $\widehat{a}$ its Fourier transform. Then for $1\leq p_2\leq p_1 \leq \infty$ we have:
\begin{align*}
&\text{Supp }\widehat{a} \subset 2^k\mathcal{B}  \ \Longrightarrow \ 
\norm {\partial^{\alpha}_{x }a}_{L^{p_1} } \lesssim 2^{k\big( |\alpha| + 2\big( \frac{1}{p_2}- \frac{1}{p_1}\big) \big)} \norm {a}_{L^{p_2} }, \\ 
&\text{Supp }\widehat{a} \subset 2^k\mathcal{C}  \ \Longrightarrow \
\norm {a}_{L^{p_1} } \lesssim 2^{-kN} \sup_{|\alpha|=N}\norm {\partial^{\alpha}_{}a}_{L^{p_1} } .
\end{align*}
\end{lemmasec}
\noindent In the following proposition, we collect some useful properties and results related to the spaces defined above, the reader can see \cite{Chemin1,Chikami,handbook,Haroune1,Haroune2} for more details.
\begin{propsec} \label{prop.properties Besov}
Let $( \delta,s)$ be in $\mathbb{R}\times \mathbb{R}^*_- $, and $p,r,\rho,m$ be in $[1,\infty]$,
\begin{itemize}
\item for $u \in B^\delta_{p,r}$, there exists some sequence $(c_{j,r})_{r\in \mathbb{Z}}$ such that
 $$\norm {\Delta_j u}_{L^p}\leq c_{j,r} 2^{-\delta j}\norm f_{\dot{B}^\delta_{p,r}} \; \text{,and } \; \sum_{j\in \mathbb{Z}} c_{j,r}^r \leq 1.$$
\item In the case of non positive regularity, one may replace, equivalently, $\Delta_j$ in the definitions of the Besov space by $S_j$, that is we have
$$\norm u_{\dot{B}^s_{p,r}} \approx \norm {  2^{js} \norm {S_j  u}_{L^p}}_{  \ell_j^{r}(\mathbb{Z})  }. $$
\item According to Minkoski's inequality, we have
\begin{equation}\label{minkoski}
L_T^\rho(\dot{B}^s_{p,r}) \hookrightarrow \tilde{L}_T^\rho(\dot{B}^s_{p,r}) \; \text{ if } \; \rho \leq r, \;\; \tilde{L}_T^\rho(\dot{B}^s_{p,r}) \hookrightarrow L_T^\rho(\dot{B}^s_{p,r}) \; \text{ if } \; r \leq \rho.
\end{equation}  
\item For $1\leq p \leq q \leq \infty $, and $1\leq r \leq m \leq \infty $, we have
$$ \dot{B}^{ \delta }_{p,r}(\mathbb{R}^d) \hookrightarrow \dot{B}^{ \delta - d(\frac{1}{p} -\frac{1}{q} )}_{q,m}(\mathbb{R}^d) .$$
\item In terms of Kato spaces (resp. Kato-Herz spaces), we have the following characterization of Besov spaces (resp. Fourier-Hezr spaces) of negative regularity $s<0$
\begin{equation}
\norm f_{\dot{B}^{s}_{p,r}} \approx \norm { e^{t\Delta} f }_{K^{-s}_{p,r}} .
\end{equation}
 \begin{equation}\label{caracterisation kato-besov}
  \norm f_{\widehat{B}^{ s}_{p,r}} \approx \norm { e^{t\Delta} f }_{\widehat{K}^{-s}_{p,r}} .
 \end{equation}
\end{itemize}
\end{propsec}
The following proposition, has been used in the previous section, which describes the continuity in Chemin-Lerner spaces of some Fourier multipliers
\begin{propsec}\label{prop.laplacian}
let $s$ be a real number, $(p,r)$ be in $[1,\infty]^2$, then we have, for all $f\in B^s_{p,r}(\mathbb{R}^d)$, for all $k\in [0,2]$
\begin{equation}\label{inequality 1.prop laplace}
\norm {(Id-\Delta)^{-1} \nabla \times f}_{\tilde{L}^\rho (\dot{B}^{s+k-1}_{p,r})}\lesssim \norm f_{\tilde{L}^\rho( \dot{B}^s_{p,r})}.
\end{equation}
\begin{equation}\label{inequality 2.prop laplace}
\norm {(Id-\Delta)^{-1} \Delta  f}_{\tilde{L}^\rho(\dot{B}^{s+m-2}_{p,r})}\lesssim \norm f_{\tilde{L}^\rho(\dot{B}^s_{p,r})}.
\end{equation}
\end{propsec}
\noindent The proof of proposition \ref{prop.laplacian} is based on the Bernstein lemma and following one
\begin{lemmasec}\label{lemma multiplicateur fourier}
Let $\mathcal{C}$ be an annulus in $\mathbb{R}^d$, $m\in \mathbb{R}$, and $k$ be the integer part of $1+\frac{d}{2}$ $ (k\bydef [1+\frac{d}{2} ])$. Let $\sigma$ be $k$-times differentiable function on $\mathbb{R}^* $ such that for all $\alpha \in \mathbb{N}^d$ with $|\alpha| \leq k$, there exists $C_\alpha$ satisfying
$$ \forall \xi \in \mathbb{R}^d \;,\; |\partial^\alpha \sigma(\xi)| \leq C_\alpha (1+ |\xi|^2)^{m }|\xi|^{-\alpha}.$$
Then, there exists $C$ depends only on $C_\alpha$ such that for any $p\in [1,\infty] $ and $\lambda>0$, we have, for any $u\in L^p$ satisfying $ supp(\widehat{u}) \subset \lambda \mathcal{C}$,
$$ \norm {\sigma(D) u}_{L^p} \lesssim C (1+\lambda^2)^m \norm u_{L^p}, \text{ with } \sigma(D) u \bydef \mathcal{F}^{-1} (\sigma \widehat{u}).$$
\end{lemmasec}
 \begin{proof}
 Following the proof of lemma 2.2 from \cite{Chemin1}, seen that $ supp(\widehat{u}) \subset \lambda \mathcal{C} $, we can write
 \begin{align*}
 &\sigma(D) u= \lambda^d K_\lambda (\lambda \cdot ) \star u \text {\; with }\\
 &K_\lambda (x  ) \bydef (2\pi)^{-d} \int_{\mathbb{R}^d} e^{i(x|\xi)}\tilde{\phi}(\xi)\sigma(\lambda\xi) d\xi , 
\end{align*}  
for some smooth function $\tilde{\phi} $ supported in an annulus and having value $1$ in $\mathcal{C}$.\\
Let $M$ be the integer part of $(1+|x|^2)^M$. We have
\begin{align*}
(1+|x|^2)^M |K_\lambda(x)| &= \bigg|\int_{\mathbb{R}^d} e^{i(x|\xi)} \big ( Id-\Delta_\xi \big)^M\tilde{\phi}(\xi)\sigma(\lambda\xi) d\xi\bigg| \\
&= \bigg|\sum_{|\alpha|+ |\beta|\leq 2 M} c_{\alpha,\beta} \lambda^{|\beta|} \int_{supp(\tilde{\phi})} e^{i(x|\xi)} \partial ^{\alpha}\tilde{\phi}(\xi)\partial ^{\beta}\sigma(\lambda\xi)d\xi\bigg|\\
&\lesssim C (1+\lambda^2)^m.
\end{align*}
As $2M>d$, we deduce that 
$$\norm {K_\lambda}_{L^1} \leq C (1+\lambda^2)^m.$$
Thus Young's inequality conclude the proof of the desired inequality.
 \end{proof}
 \begin{proof}[\textbf{\textit{Proof of proposition \ref{prop.laplacian}}}]
 According to lemma \ref{lemma multiplicateur fourier}, and Bernstein lemma, we have 
$$\norm {(Id-\Delta)^{-1} \nabla \times (\Delta_j f)}_{L^1L^p_x}\lesssim \frac{2^j}{1+2^{2j}} \norm {\Delta_j f}_{L^1L^p_x} ,$$
$$ 
\norm {(Id-\Delta)^{-1} \Delta  (\Delta_j f)}_{L^1L^p_x}\lesssim \frac{2^{2j}}{1+2^{2j}}\norm {\Delta_j f}_{L^1L^p_x}.
$$ 
The result follows from the fact that, for all $k\in [0,2]$ we have
$$ 2^{jk}\lesssim 1+2^{2j }.$$
Proposition \ref{prop.laplacian} is then proved.
 \end{proof}
 The following fixed point argument has been used to prove corollary \ref{corollary local existence}, the proof of which can be found for instance in \cite{handbook}
 \begin{lemmasec}\label{lemma fixed point3}
 Let $X$ be a Banach space, $L$ a linear operator from $X$ to $X$, with norm equals to $\lambda<1$, and let $B$ be a bi-linear operator mapping from $X\times X$ in $X$, with norm $ \norm {B}=\gamma$, then for all $y\in X$ such that
 $$ \norm y_X < \frac{(1-\gamma)^2}{4\gamma},$$
 the equation $$ x= y + L(x) + B(x,x),$$
 has a unique solution in the ball $B_X(0,\frac{1-\lambda}{2\gamma})$.
 \end{lemmasec}
\noindent Finally we recall a result concerning the smoothing effect of the Heat Kernel, one may see \cite{Chemin1,Danchin3} for more details.
\begin{lemmasec}\label{lemma heat kernel}
Let $(s,p,r)\in \mathbb{R} \times [1,\infty]^2$ and the operators $\mathcal{T} $ and $\mathcal{T}_0$ be given by
$$T_0 \mathcal{K}_0(t,x) \bydef e^{t\Delta}\mathcal{K}_0(x),$$
$$T\mathcal{K}(t,x) \bydef \int_0^t e^{\Delta (t-s)} \mathcal{K}(s,x) ds, $$
then,
\begin{equation}\label{heat kernel 1}
\norm {T_0 \mathcal{K}_0 }_{\mathscr{L}(\dot{B}^{s}_{p,r})} \lesssim \norm {\mathcal{K}_0}_{\dot{B}^{s}_{p,r}},
\end{equation}
and 
\begin{equation}\label{heat kernel 2}
\norm {T_0 \mathcal{K}_0 }_{\mathscr{L}(\dot{B}^{s}_{p,r})} \lesssim \norm { \mathcal{K}}_{\tilde{L}^1_T( \dot{B}^{s}_{p,r})}.
\end{equation}
\end{lemmasec}


\begin{thebibliography}{}

 \bibitem{abdelhamid}
H. M.Abdelhamid, Y. Kawazura, and Z. Yoshida, \textit{Hamiltonian formalism of extended magnetohydrodynamics,}. J. Phys. A: Math. Theor., 48 (2015), 235502.

\bibitem{Chemin1}
H.Bahouri,J.-Y.Chemin,R.Danchin, \textit{Fourier Analysis and Nonlinear Partial Differential Equations}, Springer-Verlag, Berlin-Heidelberg-Newyark, 2011.
 
\bibitem{handbook}
M.Cannone \textit{HANDBOOK OF MATHEMATICAL FLUID DYNAMICS, VOLUME III}. Elsevier B.V., 68 (143) (2005), 161-244.

 \bibitem{Chikami}
N,Chikami, \textit{On Gagliardo–Nirenberg type inequalities in Fourier–Herz spaces}.Journal of Functional Analysis 275, 2018, 1138–1172.

\bibitem{Danchin-Jin}
R.Danchin, J.Tan, \textit{On the well-Posedness of the Hall-Magnetohydrodynamics system in critical spaces,} arXiv:1911.03246. (2019)

\bibitem{Liu-Jin}
R.Danchin, J.Tan, \textit{Global well-posedness for the Hall-Magnetohydrodynamic system in larger critical Besov spaces,} arXiv:2001.02588v1, (2020).

 \bibitem{Danchin2}
R.Danchin, M.Paicu, \textit{Global existence results for the anisotropic Boussinesq system in dimension two}.Mathematical Models and Methods in Applied Sciences, 2011, vol. 21, no 03, p. 421-457.

 \bibitem{Danchin3}
R.Danchin, M.Paicu, \textit{Les th\'eorèmes de Leray et de Fujita-Kato pour le syst\'eme de Boussinesq partiellement visqueux}. Bulletin de la Société Mathématique de France, Tome 136 (2008) no. 2, p. 261-309
 
\bibitem{Planchon-assympto}
I.Gallagher, D.Iftime, F.Planchon, \textit{Asymptotics and stability for global solutions to the Navier-Stokes equations}. Annales de l'Institut fourier. 2003. p. 1387-1424.

\bibitem{Haroune2}
H.Houamed, \textit{About some possible blow-up conditions for the 3-D Navier-Stokes equations,} arXiv:1904.12485. (2019)

\bibitem{Haroune1}
H.Houamed, P.Dreyfuss, \textit{Uniqueness result for the 3-D Navier-Stokes-Boussinesq equations with horizontal dissipation,} arXiv:1904.00437. (2019)

\bibitem{Pei}
A.Larios,Y.Pei, \textit{On the local well-posedness and a Prodi-Serrin type regularity criterion of the three-dimensional MHD-Boussinesq system
without thermal diffusion}. Journal of Differential Equations, Volume 263, Issue 2, 15 July 2017, Pages 1419-1450.

\bibitem{Leray1}
J. Leray, \textit{Essai sur le mouvement d'un liquide visqueux emplissant l'espace, Acta Mathematica,}. 63, 1933,pages 193-248.

\bibitem{Zhao}
Y. Fukumoto, X.Zhao \textit{Well-posedness and large time behavior of solutions for the electron inertial Hall-MHD system}. Advances in Differential Equations (2019), 161-244.

\bibitem{Giga-Yoshida}
Y. Giga and Z.Yoshida, \textit{On the equations of the two-component theory in magnetohy-drodynamics,}. Comm. Partial Differential Equations, 9 (1984), 503-522.

\bibitem{Giga-Ibarhim-AL}
Y. Giga, S. Ibrahim, S. Shen, and T. Yoneda, \textit{Global well posedness for a two-
fluid model,}. Differential Integral Equations, 31 (2018), 187-214.

\bibitem{Agapito-Schonbek}
R. Agapito and M. Schonbek, \textit{Non-uniform decay of MHD equations with and without
magnetic diffusion,}. Comm. PDE, 32 (2007), 1791-1812.

\bibitem{Chae-Wan-Wu}
D. Chae, R. Wan, and J. Wu, \textit{Local well-posedness for Hall-MHD equations with
fractional magnetic diffusion,}. J. Math. Fluid Mech., 17 (2015), 627-638.

\bibitem{Wan-Zhou}
R. Wan and Y. Zhou, \textit{On global existence, energy decay and blow-up criteria for the Hall-MHD system,}. J. Differential Equations, 259 (2015), 5982-6008.

  
  
 





 





 
\end{thebibliography}
\end{document}